\documentclass[12pt]{amsart}

\usepackage{amsmath}
\usepackage{amssymb}
\usepackage{amsthm}
\usepackage{url}

\theoremstyle{plain}
\newtheorem{theorem}{Theorem}[section]
\newtheorem{corollary}[theorem]{Corollary}
\newtheorem{lemma}[theorem]{Lemma}

\newtheorem{proposition}[theorem]{Proposition}

\newtheorem{claim}[theorem]{Claim}

\theoremstyle{definition}
\newtheorem{definition}[theorem]{Definition}
\newtheorem{remark}[theorem]{Remark}

\newtheorem{notation}[theorem]{Notation}

\theoremstyle{remark}

\def\cf{\mathop{\mathrm{cf}}}
\def\dom{\mathop{\mathrm{dom}}}
\def\On{\mathop{\mathrm{On}}}
\def\Rk{\mathop{\mathrm{ER}}}

\title{The Hanf number for amalgamation of coloring classes}

\author{Alexei Kolesnikov}
\address{Department of Mathematics, Towson University \\ 8000 York Rd., Towson, MD 21252}
\email{akolesnikov@towson.edu}

\author{Chris Lambie-Hanson}
\address{Einstein Institute of Mathematics, Givat Ram Campus, Hebrew University of Jerusalem \\ Jerusalem 91904 Israel}
\email{clambiehanson@math.huji.ac.il}

\begin{document}

\begin{abstract}
We study amalgamation properties in a family of abstract elementary classes that we call coloring classes. The family includes the examples previously studied in~\cite{BKS}. We establish that the amalgamation property is equivalent to the disjoint amalgamation property in all coloring classes; find the Hanf number for the amalgamation property for coloring classes; and improve the results of~\cite{BKS} by showing, in ZFC, that the (disjoint) amalgamation property for classes $K_\alpha$ studied in that paper must hold up to $\beth_\alpha$ (only a consistency result was previously known).
\end{abstract}

\maketitle

\section*{Introduction}

Amalgamation in an abstract elementary class is a frequently made assumption in 
various structure results; for example, the amalgamation property is a standing assumption in Chapters~8--15 of~\cite{Ba} and is an assumption in~\cite{GrVaDupcat,Lessmannupcat,Sh394}. However, the exact strength of this assumption is still unknown. 
In particular, it is an open problem, posed as Conjecture 9.3 in~\cite{Gro}, whether there is a Hanf number for amalgamation. An expanded version of this question also appears on the list of open problems in~\cite{Ba} (Problem~D.3).

More precisely, suppose that $\mathfrak{K}$ is a family of abstract elementary classes. 
Is there a cardinal $\lambda(\mathfrak{K})$ such that for every $K\in \mathfrak{K}$, the class $K$ has the amalgamation property in some $\mu > \lambda(\mathfrak{K})$ if and only if $K$ has the amalgamation property in all $\mu > \lambda(\mathfrak{K})$? If the answer is yes, then we will call the least such cardinal the \emph{Hanf number of $\mathfrak{K}$ for amalgamation}. A typical example of a family $\mathfrak{K}$ is the collection of all abstract elementary classes $K$ such that the L\"owenheim--Skolem number of $K$ has a fixed size $\kappa$. A general discussion of the term ``the Hanf number for a property $P$'' is in Chapter~4 of~\cite{Ba}.

Partial advances were made in~\cite{BKS}, where, for each infinite cardinal $\kappa$ and each $\alpha<\kappa^+$, a family of examples of abstract elementary classes $K_{\alpha+2}$ in a language $L_\alpha$, $|L_\alpha|=\kappa$, were given such that $K_\alpha$ has the \emph{disjoint} amalgamation property up to $\aleph_\alpha$, but for which the disjoint amalgamation property eventually fails. In fact, none of the classes $K_\alpha$ has arbitrarily large models. It was established that, consistently with $\aleph_\alpha <\beth_\alpha$, the disjoint amalgamation property for $K_{\alpha+2}$ holds up to $\beth_\alpha$. Thus, for the family of all AECs in a language of cardinality $\kappa$, the Hanf number for \emph{disjoint} amalgamation, if it exists, has to be at least $\aleph_{\kappa^+}$ and, consistently with $\aleph_{\kappa^+} <\beth_{\kappa^+}$, has to be at least $\beth_{\kappa^+}$. An error in Proposition~3.6 of~\cite{BKS} was pointed out by Mirna Dzamonja, but it was shown in~\cite{BKSfix} that the main consistency result (Theorem~3.10) still holds. In this paper, we substantially improve the result by obtaining the conclusion of Theorem~3.10 in ZFC. The paper~\cite{BKS} did not address the amalgamation property.

A recent paper~\cite{BKL} addresses a more ambitious problem of classifying possible \emph{amalgamation spectra} for abstract elementary classes. The paper presents a family of abstract elementary classes $K_k$, $k<\omega$, each axiomatized by a complete $L_{\omega_1,\omega}$-sentence and such that $K_k$ has disjoint amalgamation in $\aleph_0,\dots,\aleph_{k-2}$, but has neither amalgamation nor disjoint amalgamation in $\aleph_{k-1}$. Amalgamation  trivially holds in $\aleph_k$, since every model of $K_k$ of that size is maximal (so there are no models of cardinality greater than $\aleph_k$).

In the present paper, we introduce a family of abstract elementary classes that we call \emph{coloring classes}. This family includes the examples of~\cite{BKS}, but does not include the examples of~\cite{BKL}. Our main results include the following, which will be made more precise in the relevant later sections.

\begin{enumerate}
	\item{For an arbitrary coloring class, the disjoint amalgamation property is equivalent to the amalgamation property. In particular, the results of \cite{BKS} also apply to the problem of finding the Hanf number for amalgamation.}
	\item{We improve the results of \cite{BKS} by showing \emph{in ZFC} that, for the classes $K_{\alpha+2}$ studied in \cite{BKS}, the (disjoint) amalgamation property holds up to $\beth_\alpha$.}
	\item{For the collection of all coloring classes in a language of a fixed size $\kappa$, the Hanf number for (disjoint) amalgamation is precisely $\beth_{\kappa^+}$.}
	\item{Given an infinite cardinal $\kappa$, for every limit ordinal $\beta < \kappa^+$ and every $k < \omega$, there is a coloring class in a language of size $\kappa$ for which amalgamation first fails somewhere between $\beth_{\beta + k}$ and $\beth_{\beta + {k+3 \choose 2}}$.}
\end{enumerate}

One of the new tools in the analysis is a rank of finite indiscernible substructures of models in a class.
The rank is implicit in the examples of~\cite{BKS}. The values of the rank of one-element structures control both the existence of arbitrarily large models and the (disjoint) amalgamation property. The method for constructing models of size $\beth_\alpha$ is also new; unlike most of the existing methods, the inductive argument uses the entire family of coloring classes rather than a single class.

We assume that the reader is familiar with the basics of abstract elementary classes (for example, the material in Part 2 of~\cite{Ba}).

\begin{notation}
If $A$ is a set, then by $[A]^n$ we denote the set of all $n$-element subsets of $A$. $[A]^{<\omega}$ denotes the set of all finite subsets of $A$. The symbol $[n]$ denotes the set $\{1,\dots,n\}$; so $[0]$ is the empty set. We let $\omega_+:=\omega\setminus\{0\}$. 

The class sequence of cardinals $\kappa_\alpha$ (these cardinals are equal to $\beth_\alpha$ for $\alpha\ge \omega^2$) is defined at the start of Section~\ref{late_amalg}.

If $M$ is a model, then $|M|$ denotes the universe of the model and $\|M\|$ denotes the cardinality of the universe.

The notion of a set $\mathcal{W}$ of allowed diagrams appears in~Definition~\ref{allowed_diag}; the symbols $\mathcal{W}^n$, $n<\omega$ are explained there. The notation of the form $\mathcal{W}_S$ is explained in~\ref{allowed_diag_pruned}; and $\mathcal{W}/\overline{w}$ is in~\ref{allowed_diag_trimmed}.
\end{notation}

\section{Coloring classes and the existence rank}

\subsection{Coloring classes}

\begin{definition}
Let $L$ be a relational language whose set of relation symbols is $\mathcal{R}=\bigcup_{1\le n<\omega} \mathcal{R}_n$, where, for each $1\le n<\omega$, $\mathcal{R}_n$ is a non-empty set of $n$-ary relation symbols. An $L$-structure $M$ is called an \emph{$L$-coloring structure} if there is a function $c_M:[|M|]^{<\omega}\to \mathcal{R}$ such that $c_M(\{a_1,\dots,a_n\})=P$ if and only if $M\models P(a_1,\dots,a_n)$.

If $M$ is an $L$-coloring structure given by the function $c_M$, then we refer to the function $c_M$ as the \emph{coloring function}.
\end{definition}

\begin{remark}
It is clear that $M$ is a coloring structure if and only if every $P\in \mathcal{R}_n$ is a relation on $n$-element \emph{subsets} of $|M|$ and the realizations of the relations in $\mathcal{R}_n$ partition $[|M|]^n$.
\end{remark}

Everywhere below, we fix a relational language $L$ that has at least one relation of each arity. 

\begin{definition}
Let $N$ be an $L$-coloring structure with the corresponding coloring function $c_N$ and suppose that $M$ is a substructure of $N$. We say that $M$ is \emph{$L$-monochromatic} if for each $n$, for all $A,B\in [|M|]^n$ we have $c_N(A)=c_N(B)$.

If $M$ is an $n$-element $L$-monochromatic substructure of $N$, the \emph{diagram of $M$} is the function $d_M:[n]\to \mathcal{R}$ given by $d_M(k):=c_N(A)$ for some (any) $k$-element subset $A$ of $M$. If $M$ is an infinite $L$-monochromatic substructure, then the diagram of $M$ is the function $d_M:\omega_+\to \mathcal{R}$ defined similarly: $d_M(k):=c_N(A)$ for some (any) $k$-element subset $A$ of $M$.
\end{definition}

In the language of model theory, an $L$-monochromatic substructure of $N$ is a subset of $N$ indiscernible with respect to quantifier-free formulas in $L$ and the diagram $d_M$ codes the quantifier-free type of the indiscernible substructure. For the purposes of this paper, we find it convenient to work with functions, and hence we will be using the terminology of colorings.

\begin{definition}\label{allowed_diag}
A \emph{set of allowed diagrams} is a non-empty subset $\mathcal{W}$ of the set of functions $\{w:[n]\to \mathcal{R}\mid n<\omega, w(k)\in \mathcal{R}_k \textrm{ for }1\le k\le n\}$ such that for all $w\in \mathcal{W}$, $w:[n]\to \mathcal{R}$, if $m\le n$, then $w\restriction [m]\in \mathcal{W}$. 

We use the symbol $\mathcal{W}^n$ to denote the set of functions in $\mathcal{W}$ with domain $[n]$. Note that $\mathcal{W}^0=\{ \emptyset \}$.

Given a set $\mathcal{W}$ of allowed diagrams, the class $K(\mathcal{W})$ is the class
of all $L$-coloring structures $N$ such that $d_M\in \mathcal{W}$ for every finite monochromatic substructure $M\subset N$. 
\end{definition}

\begin{remark}
It is easy to check that, for a set $\mathcal{W}$ of allowed diagrams in a relational language $L$, the pair $(K(\mathcal{W}),\subset)$, where $\subset$ is the substructure relation, forms an abstract elementary class with countable L\"owenheim--Skolem number. Indeed, it follows from the definition of the substructure relation that all the axioms of abstract elementary classes hold, except for the union of chains axiom and the existence of a L\"owenheim--Skolem number. The latter two axioms follow since the membership of an $L$-structure in $K$ is determined by the properties of its finite substructures.
\end{remark}

\begin{definition}
If $c:[|M|]^{<\omega} \to \mathcal{R}$ is a coloring function such that the $L$-structure $M$ given by $c$ is in the class $K(\mathcal{W})$, then we call $c$ a \emph{well-coloring with respect to $\mathcal{W}$}, or simply a $\mathcal{W}$-coloring.

An abstract elementary class $\mathcal{K}$ in a relational language $L$ is a \emph{coloring class} if $\mathcal{K}=K(\mathcal{W})$ for some set of allowed diagrams $\mathcal{W}$.
\end{definition}

Note that, if $\mathcal{W}$ is a set of allowed diagrams, then $(\mathcal{W}, \subset)$ is a tree, the root of the tree is $\emptyset$, and, for $n < \omega$, $\mathcal{W}^n$ is the $n$-th level of the tree. The interesting cases will be those in which this tree is well-founded, i.e. in which there is no $d : \omega_+ \to \mathcal{R}$ such that, for all $n < \omega$, $d \restriction [n] \in \mathcal{W}$. Such a function $d$ would correspond to an infinite monochromatic structure in $K(\mathcal{W})$. With this in mind, we introduce a rank function on the elements of $\mathcal{W}$. This is simply the usual rank function on well-founded trees, but we provide an explicit definition for completeness.

\begin{definition}
Fix a set $\mathcal{W}$ of allowed diagrams in the language $L$. Define an \emph{existence rank $\Rk(w;\mathcal{W})$ with respect to $\mathcal{W}$} on the elements of $w\in \mathcal{W}$ by induction as follows. If $n<\omega$ and $w:[n]\to \mathcal{R}$ is an element of $\mathcal{W}^n$, then:
\begin{enumerate}
\item
$\Rk(w;\mathcal{W})\ge 0$;
\item
if $\alpha$ is a limit ordinal, then $\Rk(w;\mathcal{W})\ge \alpha$ provided $\Rk(w;\mathcal{W})\ge \beta$ for all $\beta<\alpha$;
\item
if $\alpha=\beta+1$, then $\Rk(w;\mathcal{W})\ge \alpha$ if there is $w'\in \mathcal{W}^{n+1}$ such that $w'\supset w$ and $\Rk(w';\mathcal{W})\ge \beta$.
\end{enumerate}
If $\Rk(w;\mathcal{W})\ge \alpha$ and $\Rk(w;\mathcal{W})\not\ge \alpha+1$, then we say $\Rk(w;\mathcal{W})=\alpha$. If $\Rk(w;\mathcal{W})\ge \alpha$ for all ordinals $\alpha$, then we say $\Rk(w;\mathcal{W})=\infty$.
\end{definition}

\begin{proposition}
\label{infinite_rank}
Let $\mathcal{W}$ be a set of allowed diagrams. The following are equivalent:
\begin{enumerate}
\item
$\Rk(w;\mathcal{W})=\infty$;
\item
$\Rk(w;\mathcal{W})\ge |L|^+$;
\item
There is an infinite monochromatic structure $M\in K(\mathcal{W})$ such that $d_M\supset w$.
\end{enumerate}
\end{proposition}

\begin{proof}
$(1)\Rightarrow (2)$ is immediate. 

$(2)\Rightarrow (3)$ will follow once we prove that for every $1\le n < \omega$ and every $w\in \mathcal{W}^n$ such that $\Rk(w_1;\mathcal{W})\ge |L|^+$, there exists a proper extension $w^*\in \mathcal{W}^{n+1}$ of $w$ such that $\Rk(w^*;\mathcal{W})\ge |L|^+$. Indeed, given $w$ as above, for every $\beta<|L|^+$, there is $u_\beta\in \mathcal{W}^{n+1}$ such that $\Rk(u_\beta;\mathcal{W})\ge \beta$. Since there are at most $|L|$ distinct such extensions, there is $w^*\in \mathcal{W}$ such that $w^*=u_\beta$ for unboundedly many $\beta<|L|^+$. But then $\Rk(w^*;\mathcal{W})\ge |L|^+$.

$(3)\Rightarrow (1)$ Suppose $N$ is an infinite monochromatic structure in $K(\mathcal{W})$. By induction on $\alpha$, one can show that $\Rk(d_M;\mathcal{W})\ge \alpha$ for all finite monochromatic $M\subset N$. This is easy, as every finite substructure of $N$ extends to a larger monochromatic well-coloring.
\end{proof}

It follows that if $\mathcal{W}$ contains an infinite-rank element, then the coloring class $K(\mathcal{W})$ contains a model of arbitrarily large size.

\subsection{Amalgamation is equivalent to disjoint amalgamation in coloring classes}

\begin{definition}
Fix a set of allowed diagrams $\mathcal{W}$ and a cardinal $\lambda$. 
A pair of $\mathcal{W}$-colorings $\{c_1,c_2\}$ is a \emph{special $(\lambda,2)$-system} if there is a set $X$ of size $\lambda$ and elements $a_1,a_2\notin X$ such that
\begin{enumerate}
\item
the domain of $c_i$ is $[X\cup \{a_i\}]^{<\omega}$ for $i=1,2$;
\item
$c_1\restriction [X]^{<\omega}= c_2\restriction [X]^{<\omega}$.
\end{enumerate}
\end{definition}

\begin{remark}
In the language of model theory, a special $(\lambda,2)$-system $\{c_1,c_2\}$ carries the following information. Each of the functions $c_i$, $i=1,2$, defines an $L$-structure $M_i\in K(\mathcal{W})$ with the universe $X\cup\{a_i\}$. The structures $M_1$ and $M_2$ contain a common substructure $M$ with the universe $X$. The coloring function of $M$ is the common restriction to $[X]^{<\omega}$ of the functions $c_1$ and $c_2$.

An inductive argument shows that a class $K=K(\mathcal{W})$ has disjoint amalgamation for models of size $\lambda$ if and only if for every special $(\lambda,2)$-system of colorings $\{c_1,c_2\}$ there is a $\mathcal{W}$-coloring $c\supset c_1\cup c_2$.
\end{remark}

We now show that, for sufficiently rich coloring classes, the amalgamation property is equivalent to the disjoint amalgamation property. 

\begin{proposition}
\label{ap_equiv_dap}
Let $L$ be a relational language such that $|\mathcal{R}_k|>1$ for infinitely many $k\ge 2$. Let $\mathcal{W}$ be a set of allowed diagrams, and suppose that, for every $w\in \mathcal{W}^1$, there is $n<\omega$ and $w_1,w_2\in \mathcal{W}^n$ such that $w\subset w_1,w_2$ and $w_1(n)\ne w_2(n)$. 

The class $K=K(\mathcal{W})$ has the amalgamation property for models of size $\lambda$ if and only if it has the disjoint amalgamation property for models of size $\lambda$.
\end{proposition}

\begin{proof}
It is clear that the disjoint amalgamation implies amalgamation (in any class); it suffices to establish that the converse holds in a coloring class.

Let $K=K(\mathcal{W})$ be a coloring class that has the amalgamation property in $\lambda$ and suppose that $L$ and $\mathcal{W}$ satisfy the assumptions of the proposition.
To establish the disjoint amalgamation property in $\lambda$, it is enough to show that for a set $X$ of size $\lambda$ any two $\mathcal{W}$-colorings $c_i:[X\cup \{a_i\}]^{<\omega} \to \mathcal{R}$, $i=1,2$, that agree on $[X]^{<\omega}$ can be extended to a $\mathcal{W}$-coloring of $[X\cup \{a_1,a_2\}]^{<\omega}$. Let $M_0\in K$ be determined by the common restriction of $c_1$, $c_2$ to $[X]^{<\omega}$, and let $M_1$ and $M_2$ be the structures determined by $c_1$ and $c_2$. We split the argument into three cases. 

\textbf{Case 1:} $c_1(\{a_1\}\cup C)\ne c_2(\{a_2\}\cup C)$ for some $C\in [X]^{<\omega}$.

If $M^*$ is the amalgam of $M_1$ and $M_2$ over $M_0$ and $f_i:M_i\to M^*$ are the corresponding embeddings, then the substructure $f_1(M_1)\cup f_2(M_2)$ of $M^*$ is the disjoint amalgam of $M_1$ and $M_2$ over $M_0$, since it cannot be the case that $f_1(a_1) = f_2(a_2)$. 

\textbf{Case 2:} $c_1(\{a_1\}\cup C)= c_2(\{a_2\}\cup C)$ for all $C\in [X]^{<\omega}$ and there is $k<\omega$ and $w\in \mathcal{W}^k$ such that $w\supset d_{\{a_1\}}$ and $w\ne d_B$ for every monochromatic $B\in [X\cup a_1]^k$. 

Note that $k\ge 2$. We can then define a well-coloring $c$ of $X \cup \{a_1,a_2\}$ amalgamating $c_1$ and $c_2$ as follows. Note that, to define $c$, we only have to specify its values on sets of the form $C \cup \{a_1, a_2\}$, where $C \in [X]^{<\omega}$. For all $0\le i \le k-2$ and all $C\in [X]^i$, let $c(C \cup \{a_1,a_2\}) = w(i+2)$. Note that, with this definition, if $B\in [X]^{\ell}$ for some $\ell > k-2$, then it is impossible for $\{a_1,a_2\}\cup B$ to be monochromatic with respect to $c$. Thus, $c$ can be extended to larger sets arbitrarily. It is easily verified that the $L$-structure given by $c$ is in $K(\mathcal{W})$.

\textbf{Case 3:} $c_1(\{a_1\}\cup C)= c_2(\{a_2\}\cup C)$ for all $C\in [X]^{<\omega}$ and, for every $k<\omega$ and every $w\in \mathcal{W}^k$ such that $w\supset d_{\{a_1\}}$, there is a monochromatic $B\in [X\cup a_1]^k$ with $d_B=w$. 

We modify the coloring $c_1$ to get a different coloring $c'_1$ of $X\cup\{a_1\}$ 
so that the pair $\{c'_1,c_2\}$ satisfies the assumptions of Case~1. The coloring $c'_1$ will coincide with $c_1$ for all but one subset of $X\cup\{a_1\}$. Fix $n < \omega$ and distinct $w_1, w_2 \in \mathcal{W}^n$ such that $w_1, w_2 \supset d_{\{a_1\}}$ and $w_1(n) \not= w_2(n)$. Find $B_1, B_2 \in [X]^{n-1}$ such that, for $i=1,2$, $d_{(\{a_1\} \cup B_i)}=w_i$. By the assumptions of the Proposition, there is $k>2n-1$ such that $|\mathcal{R}_k|>1$. Let $C\supseteq \{a_1\}\cup B_1\cup B_2$ be a $k$-element subset. Define $c'_1(C)$ to be a distinct color from $c_1(C)$ and let $c'_1(D):=c_1(D)$ for all finite $D\ne C$. The coloring $c'_1$ is easily seen to be a $\mathcal{W}$-coloring.
Indeed, no monochromatic subset can contain $C$. Now amalgamation in $\lambda$ implies, by Case~1, that $c'_1$ and $c_2$ can be extended to a coloring $c'$. The coloring $c'$ can easily be turned into a $\mathcal{W}$-coloring extending $c_1$ and $c_2$ simply by appropriately changing the value on $C$. 
\end{proof}

%\begin{corollary}
%\label{ap_equiv_dap}
%Let $L$ be a relational language such that $|\mathcal{R}_k|>1$ for infinitely many $k\ge 2$. Suppose that for %every $w\in \mathcal{W}^1$ there is $n<\omega$ and $w_1,w_2\in \mathcal{W}^n$ such that $w\subset w_1,w_2$ and %$w_1(n)\ne w_2(n)$. Let $K=K(\mathcal{W})$. The class $K$ has amalgamation property for models of size $\lambda$ %if and only if it has the disjoint amalgamation property for models of size $\lambda$.
%\end{corollary}

\section{Amalgamation in one large power implies amalgamation in all powers}

The main goal of this section is to prove the following result.

\begin{theorem}
\label{Hanf_upside}
Let $K=K(\mathcal{W})$ be a coloring class, let $\lambda\ge \beth_{|L|^+}$, and suppose that $K_\lambda$ is non-empty. Then $K$ has no maximal models and, if $K_\lambda$ has the (disjoint) amalgamation property, then $K_\mu$ has the (disjoint) amalgamation property for all $\mu\ge \beth_{|L|^+}$.
\end{theorem}

In particular, the above theorem shows that the Hanf number (for the existence of models) is $\beth_{|L|^+}$ for the family of coloring classes in the language $L$. Note that, for general abstract elementary classes, the Hanf number is $\beth_{(2^{|L|})^+}$.

We analyse a coloring class with the set of allowed diagrams $\mathcal{W}$ by examining coloring classes with smaller, ``pruned,'' sets of allowed diagrams.

\begin{notation}\label{allowed_diag_pruned}
If $S\subseteq \mathcal{W}$, let $\mathcal{W}_S$ denote the set $\{ w \in \mathcal{W} \mid$ for some $u \in S$, $w\subseteq u$ or $u\subseteq w \}$.
\end{notation}

\begin{remark}
	In the language of trees, forming $\mathcal{W}_S$ from $\mathcal{W}$ amounts to pruning the tree associated with $\mathcal{W}$ to leave only those nodes that are comparable with elements of $S$.
\end{remark}

\begin{proposition}
\label{pruned_rank}
Let $\mathcal{W}$ be a set of allowed diagrams in $L$.
\begin{enumerate}
\item
If $u,w\in \mathcal{W}$ and $u\subseteq w$, then $\Rk(w;\mathcal{W}_{\{u\}}) = \Rk(w;\mathcal{W})$.
\item
Suppose $\Rk(\emptyset;\mathcal{W}) = \alpha + k$, where $\alpha$ is a limit ordinal and $k < \omega$. Then there is $\overline{w}\in \mathcal{W}^k$ such that $\Rk(\overline{w};\mathcal{W}) = \alpha$. Moreover, there are disjoint sets $\{S_i \subset (\mathcal{W}_{\{\overline{w}\}})^{k+1}\mid i < \cf(\alpha) \}$ such that $\Rk(\emptyset;\mathcal{W}_{S_i}) = \alpha+k$ for every $i < \cf(\alpha)$.
\end{enumerate}
\end{proposition}

\begin{proof}
	(1) is immediate from the definitions, so we only provide a proof of (2). We first find $\overline{w} \in \mathcal{W}^k$ with $\Rk(\overline{w}; \mathcal{W}) = \alpha$. Using the definition of the existence rank and the fact that $\Rk(\emptyset; \mathcal{W}) = \alpha + k$, we can easily recursively define a sequence $\langle w_j \mid j \leq k \rangle$ such that:
	\begin{itemize}
		\item{$w_0 = \emptyset$;}
		\item{for all $j \leq k$, $w_j \in \mathcal{W}^j$;}
		\item{for all $j_0 < j_1 \leq k$, $w_{j_0} \subset w_{j_1}$;}
		\item{for all $j \leq k$, $\Rk(w_j; \mathcal{W}) = \alpha + k - j$.}
	\end{itemize}
	Then, letting $\overline{w} = w_k$, we see that $\overline{w}$ is as required in the statement of the proposition.

	For the `moreover' clause, note that, since $\Rk(\overline{w}; \mathcal{W}) = \alpha$ and $\alpha$ is a limit ordinal, we can find $\{u_\eta \mid \eta < \cf(\alpha)\}$ such that:
	\begin{itemize}
		\item{for all $\eta < \cf(\alpha)$, $u_\eta \in \mathcal{W}^{k+1}$ and $\overline{w} \subset u_\eta$;}
		\item{letting $\alpha_\eta = \Rk(u_\eta; \mathcal{W})$ for all $\eta < \cf(\alpha)$, we have that $\langle \alpha_\eta \mid \eta < \cf(\alpha) \rangle$ is a strictly increasing sequence, cofinal in $\alpha$.}
	\end{itemize}
	Partition $\cf(\alpha)$ into disjoint cofinal sets $\{A_i \mid i < \cf(\alpha)\}$ and, for $i < \cf(\alpha)$, let $S_i = \{u_\eta \mid \eta \in A_i\}$. Then, for all $i < \cf(\alpha)$, it is easy to check by the definitions that $S_i \subset (\mathcal{W}_{\overline{w}})^{k+1}$ and $\Rk(\emptyset; \mathcal{W}_{S_i}) = \alpha + k$. 
\end{proof}

The following lemma establishes that, if $\Rk(\emptyset; \mathcal{W})$ is bounded, then there is a bound on the size of the models in $K(\mathcal{W})$.

\begin{lemma}
\label{rank-size}
Let $\mathcal{W}$ be a set of allowed diagrams in $L$, and let $K=K(\mathcal{W})$ be the corresponding coloring class. Let $w:[n]\to \mathcal{R}$ be an element of $\mathcal{W}$ such that $\Rk(w;\mathcal{W}) < \beta+k$, where $\beta$ is either a limit ordinal or 0 and $k$ is a natural number. Then for any $M\in K(\mathcal{W}_{\{w\}})$ we have $\|M\|\le \beth_{\beta+nk+k(k-1)/2}(|L|)$.

In particular, if $\Rk(\emptyset;\mathcal{W})=\alpha<\infty$, $\alpha=\beta+k$ for a limit ordinal $\beta$ and a natural number $k$, then any model of $K(\mathcal{W})$ has size at most $\beth_{\beta+{k+1 \choose 2}}(|L|)$.
\end{lemma}

\begin{proof}
We use induction on $\alpha=\beta+k$. If $\Rk(w;\mathcal{W}) < 1$, then $M$ has at most $n$ elements. If $\alpha$ is a limit ordinal and $\Rk(w;\mathcal{W}) < \alpha$, then $\Rk(w;\mathcal{W}) < \gamma$ for some successor ordinal $\gamma<\alpha$, and the conclusion follows from the inductive hypothesis.

It remains to consider the successor case. Suppose for contradiction that $w:[n]\to \mathcal{R}$ is a function in $\mathcal{W}$ such that $\Rk(w;\mathcal{W}) < \beta+k+1$, but there is $M\in K(\mathcal{W}_{\{w\}})$, with coloring function $c_M$, such that $\|M\|\ge (\beth_{\beta+(k+1)n+k(k+1)/2}(|L|))^+$.  By the Erd\H{o}s--Rado theorem, there is a substructure $M_1\subset M$ of size $(\beth_{\beta+nk+k(k+1)/2}(|L|))^+=(\beth_{\beta+(n+1)k+k(k-1)/2}(|L|))^+$ and $P\in \mathcal{R}_{n+1}$ such that $c_M(A)=P$ for all $A \in [M_1]^{n+1}$. Now we extend $w$ to the function  $w':[n+1]\to \mathcal{R}$ by letting $w'(n+1):=P$. Note that $w'\in\mathcal{W}$ and that $M_1\in K(\mathcal{W}_{\{w'\}})$.
By the induction hypothesis, $\Rk(w';\mathcal{W})\ge \beta+k$, and thus $\Rk(w;\mathcal{W})\ge \beta+k+1$, a contradiction.

For the last statement, the assumption implies that $\Rk(\emptyset;\mathcal{W}) < \beta+k+1$, and the bound established above (with $n=0$) gives the needed result.
\end{proof}

\begin{corollary}
\label{large_model_corollary}
Let $K=K(\mathcal{W})$ be a coloring class, let $\lambda\ge \beth_{|L|^+}$, and suppose that $K_\lambda$ is non-empty. Then 
\begin{enumerate}
\item
$K$ has models in all powers;
\item
moreover, $K$ has no maximal models; and
\item
for any triple of models $M_1\subset M_2,M_3$ of $K$ and any set $X$ disjoint from $|M_2|\cup |M_3|$, there is a triple of models $N_1\subset N_2,N_3$ in $K$ such that $M_i\subset N_i$ and $|N_i|=|M_i|\cup X$ for $i=1,2,3$.
\end{enumerate}
\end{corollary}

\begin{proof}
Since $K$ has a model of size $\beth_{|L|^+} = \beth_{|L|^+}(|L|)$, the rank $\Rk(\emptyset;\mathcal{W})$ is at least $|L|^+$ by Lemma~\ref{rank-size}. Using Proposition~\ref{infinite_rank}, we get an infinite monochromatic structure $M\in K(\mathcal{W})$ with diagram $d:=d_M$. Therefore, the monochromatic structure on $\mu$ with the diagram $d$ is a model in $K_\mu$.

The second statement follows from the third by taking $M_1=M_2=M_3$ and $X$ a set containing elements not in $|M_1|$. So we prove the third statement. 

Take  $M_1\subset M_2,M_3$ and $X$ as in (3). For each $i=1,2,3$, define the coloring function $c_{N_i}$ on $|N_i|:=|M_i|\cup X$ as follows. For $A\in [|N_i|]^n$, if $A\subset |M_i|$, then $c_{N_i}(A):=c_{M_i}(A)$; otherwise, let $c_{N_i}(A):=d(n)$. 

It remains to check that the resulting coloring functions give the needed models in $K(\mathcal{W})$. We first fix $i\in \{1,2,3\}$, and show that $N_i\in K(\mathcal{W})$. Take an arbitrary non-empty finite monochromatic substructure $A$ of $N_i$ and let $n=\|A\|$. If $A\subset M_i$, then the diagram of $A$ is in $\mathcal{W}$ since $M_i\in K(\mathcal{W})$. Suppose now that the universe of $A$ contains elements of the set $X$. We claim that in that case, the diagram $d_A$ of $A$ is equal to $d\restriction [n]$ (recall that $d$ is the diagram of the infinite structure in $K$ that was used to define $N_i$). Indeed, for any $k\in [n]$, there is a substructure $A_k$ of $A$ such that $|A_k|\cap X\ne \emptyset$. By definition, we have $c_{N_i}(A_k)=d(k)$. Since $A$ is monochromatic, every $k$-element substructure of $A$ has the same color; thus $d_A(k)=d(k)$. Since $d\restriction [n]\in \mathcal{W}$, the diagram of $A$ is in $\mathcal{W}$, so $N_i\in K$.

Now we check $N_1\subset N_i$, $i=2,3$. For this, it suffices to check that every finite substructure $A$ of $N_1$ is a substructure of $N_i$, $i=2,3$. We do this by induction on $\|A\|$ (in this case, we may start with the empty substructure). Take $A\subset N_1$, $\|A\|=n$, and suppose that all proper substructures of $A$ are substructures of $N_2$, $N_3$. If $A\subset M_1$, then $A\subset M_2,M_3$ and it immediately follows from the definitions that $A\subset N_2,N_3$. If $A\not\subset M_1$, then $|A|\cap X\ne \emptyset$, so $d_{N_i}(|A|)=d(n)$ for $i=1,2,3$. Since all proper substructures of $A$ are substructures of $N_2$, $N_3$ by the induction hypothesis, it now follows that $A\subset N_2,N_3$.

The remaining properties follow directly from the definitions.
%It is routine to check that the resulting coloring functions give the needed models in $K$.
\end{proof}

\begin{lemma}
\label{large_dap_rank}
Let $K=K(\mathcal{W})$ be a coloring class, let $\lambda\ge \beth_{|L|^+}$, and suppose that $K_\lambda$ is non-empty and has the disjoint amalgamation property. Then for every $w\in \mathcal{W}$, we have $\Rk(w;\mathcal{W})=\infty$.
\end{lemma}

\begin{proof}
Suppose that there is $w\in \mathcal{W}$ with $\Rk(w;\mathcal{W})=\alpha<\infty$. Then by Lemma~\ref{rank-size}, the class $K(\mathcal{W}_{\{w\}})$ does not have a model of size $\beth_{|L|^+}$. Therefore, there is some $\kappa<\beth_{|L|^+}$ and models $M_1\subset M_2,M_3$ in $K(\mathcal{W}_{\{w\}})$ of size $\kappa$ such that $M_2$ and $M_3$ cannot be disjointly amalgamated over $M_1$. By Corollary~\ref{large_model_corollary}(3), we can find 
models $N_i\in K$, $i=1,2,3$ of size $\lambda$. Then the disjoint amalgam of $N_2$ and $N_3$ over $N_1$ gives the amalgam of $M_2$ and $M_3$ over $M_1$, a contradiction.
\end{proof}

\begin{proof}[Proof of Theorem~\ref{Hanf_upside}]
The class contains no maximal models by Corollary~\ref{large_model_corollary}. 

By Proposition~\ref{ap_equiv_dap}, it is enough to establish that the disjoint amalgamation property of $K_\lambda$, $\lambda\ge \beth_{|L|^+}$, implies the disjoint amalgamation for $K_\mu$, for any $\mu\ge \beth_{|L|^+}$.
If $K_\lambda$ has the disjoint amalgamation, then by Lemma~\ref{large_dap_rank} we have $\Rk(w;\mathcal{W})=\infty$ for every $w\in \mathcal{W}^1$. By Proposition~\ref{infinite_rank}, we have that for every $w\in \mathcal{W}^1$ there is an infinite monochromatic structure $M_w\in K$ with the diagram $d_w\supset w$. 

Take an arbitrary $\mu\ge \beth_{|L|^+}$. Given a special $(\mu,2)$-system $\{c_1,c_2\}$ of colorings, $c_i:[X\cup \{a_i\}]^{<\omega}\to \mathcal{R}$, if $c_1(a_1)\ne c_2(a_2)$, then the $\mathcal{W}$-coloring $c\supset c_1\cup c_2$ can be defined on finite sets of the form $C\cup \{a_1,a_2\}$ in an arbitrary way. If $c_1(a_1)=c_2(a_2)$, then we find an infinite monochromatic structure with diagram $d$ such that $d(1)=c_i(a_i)$ and define $c(C\cup\{a_1,a_2\}):=d(|C|+2)$. It is easy to check that the resulting coloring function $c$ is a $\mathcal{W}$-coloring.
\end{proof}

\section{Amalgamation may fail late}
\label{late_amalg}

In the previous section, assuming $\Rk(\emptyset;\mathcal{W}) < \infty$, we established an upper bound on the size of a maximal model of $K(\mathcal{W})$ as well as an upper bound on the power in which amalgamation fails provided there is $w \in \mathcal{W}$ such that $\Rk(w;\mathcal{W}) < \infty$.

In this section, we establish a lower bound on both the existence of models of $K(\mathcal{W})$ and on the size of models that can be disjointly amalgamated.

Define a class sequence of cardinals $\langle \kappa_\alpha \mid \alpha \in \On \rangle$ as follows:
\begin{itemize}
\item[]
For $\alpha < \omega$, $\kappa_\alpha = \alpha$.
\item[]
If $\alpha$ is a limit ordinal, $\kappa_\alpha = \sup(\{ \kappa_\beta \mid \beta < \alpha \})$.
\item[]
If $\beta \geq \omega$ and $\alpha = \beta + 1$, $\kappa_\alpha = 2^{\kappa_\beta}$.
\end{itemize}
Note that $\kappa_\alpha = \beth_\alpha$ for $\alpha \geq \omega^2$.

The main result of this section is the following theorem.

\begin{theorem}
\label{Hanf_downside}
Suppose $\Rk(\emptyset;\mathcal{W}) \geq \alpha$. Then 
\begin{enumerate}
\item
there is $M \in K(\mathcal{W})$ such that $\|M\| \geq \kappa_\alpha$;
\item
If $\Rk(w;\mathcal{W}) \geq \beta + 1$ for all $w\in \mathcal{W}^1$, then $K(\mathcal{W})$ has disjoint amalgamation for models of size $\lambda$ for all $\lambda\le \kappa_\beta$.
\end{enumerate}
\end{theorem}

The strategy will be as follows. We first establish the existence of models of size $\kappa_\alpha$ (and thus, the existence of models in all smaller powers) for a coloring class given by a set $\mathcal{W}$ such that $\Rk(\emptyset;\mathcal{W})\ge \alpha$. We will then use the existence result to show that disjoint amalgamation holds.

\begin{notation}
\label{allowed_diag_trimmed}
Suppose $L$ is a relational language with the set of relation symbols $\mathcal{R}=\bigcup \mathcal{R}_n$ and $\mathcal{W}$ is a set of allowed diagrams in $L$. Let $\overline{w}\in \mathcal{W}^k$, $k\ge 1$, be a fixed element. We define a new relational language $L/\overline{w}$ and a set of allowed diagrams $\mathcal{W}/\overline{w}$ in $L/\overline{w}$ as follows. Let $(\mathcal{R}/\overline{w})_n :=\mathcal{R}_{n+k}$ for $1\le n<\omega$. If $w\supset \overline{w}$ is a function with domain $[k+n]$, let $w/\overline{w}$ denote the function $i\in [n]\mapsto w(k+i)$. Finally, let 
$$
\mathcal{W}/\overline{w} :=\{w/\overline{w}\mid w\supseteq \overline{w}, w \in \mathcal{W}\}.
$$
\end{notation}

\begin{remark}
	This operation is most easily considered by thinking of sets of allowed diagrams as trees. If $\mathcal{W}$ is a set of allowed diagrams, $1 \leq k < \omega$, and $\overline{w} \in \mathcal{W}^k$, then $\mathcal{W}/\overline{w}$ as a tree is isomorphic to the tree whose root is $\overline{w}$ and whose $n$-th level is the $(n+k)$-th level of $\mathcal{W}_{\{\overline{w}\}}$. In other words, to produce the tree associated with $\mathcal{W}/\overline{w}$ from that of $\mathcal{W}$, we first prune the tree by passing to $\mathcal{W}_{\{\overline{w}\}}$ and then chop off the stem of length $k$.
\end{remark}

\begin{proposition}
\label{quotient_rank}
Suppose $\mathcal{W}$ is a set of allowed diagrams in a relational language $L$, and suppose that $\overline{w} \in \mathcal{W}$. Then for every $w\supset \overline{w}$, if $\Rk(w;\mathcal{W})\ge \alpha$, then $\Rk(w/\overline{w}; \mathcal{W}/\overline{w})\ge \alpha$.
\end{proposition}

\begin{proof}
By induction on $\alpha$. If $\alpha=0$, this is clear. If $\alpha$ is a limit ordinal and $\Rk(w;\mathcal{W})\ge \alpha$, then $\Rk(w;\mathcal{W}) \geq \beta$ for all $\beta < \alpha$. By the induction hypothesis, $\Rk(w/\overline{w}; \mathcal{W}/\overline{w}) \geq \beta$ for all $\beta < \alpha$, so $\Rk(w/\overline{w}; \mathcal{W}/\overline{w}) \geq \alpha$. If $\alpha = \beta + 1$, let $u \in \mathcal{W}$ be such that $w \subset u$ and $\Rk(u;\mathcal{W}) \geq \beta$. $w/\overline{w} \subset u/\overline{w}$ and, by the induction hypothesis, $\Rk(u/\overline{w}; \mathcal{W}/\overline{w}) \geq \beta$, so $\Rk(w/\overline{w}; \mathcal{W}/\overline{w}) \geq \beta + 1 = \alpha$.
\end{proof}

\begin{lemma}
\label{existence_high}
Let $L$ be a relational language and $\alpha$ an ordinal, and suppose that $\mathcal{W}$ is a set of allowed diagrams in $L$ such that $\Rk(\emptyset;\mathcal{W}) \geq \alpha$. Then there is $M \in K(\mathcal{W})$ such that $\|M\| \geq \kappa_\alpha$.
\end{lemma}

\begin{proof}
We have already shown that, if $\Rk(\emptyset;\mathcal{W}) \ge |L|^+$, then $K(\mathcal{W})$ contains arbitrarily large models. Thus, it suffices to show that, if $\Rk(\emptyset;\mathcal{W}) = \alpha<|L|^+$, then $K(\mathcal{W})$ contains models of size $\kappa_\alpha$. We use induction on $\alpha$ and show that \emph{for every $L$} and \emph{for every} set of allowed diagrams, if $\Rk(\emptyset;\mathcal{W})=\alpha$, then $K(\mathcal{W})$ has a model of size $\kappa_\alpha$.

First, suppose $\alpha < \omega$. If $\Rk(\emptyset;\mathcal{W}) = \alpha$, then there is $c\in \mathcal{W}$ such that $\dom(c) = [\alpha]$. If $N$ is the monochromatic structure determined by $c$, then $N\in K(\mathcal{W})$ and $N$ has the size $\kappa_\alpha = \alpha$. Note that, in general, this is the best we can do.
% Suppose $w: n + 1 \rightarrow \mathcal{R}$ and $\mathcal{W}$ is such that every $M \in \mathcal{W}$ is determined by some initial segment of $w$. Then $Rk_{\mathcal{W}}(\emptyset) = n$, but, letting $N$ be the model determined by $w$, $N$ is the only model of size $n$ in $K(\mathcal{W})$ and there are no models of size $n + 1$, since such a model would be mono-chromatic but would not appear in $\mathcal{W}$.)

Next, suppose $\alpha$ is a limit ordinal and $\Rk(\emptyset;\mathcal{W}) = \alpha$. Let $\lambda = \cf(\alpha)$. Fix $\langle w_\gamma \mid \gamma < \lambda \rangle$ and $\langle \beta_\gamma \mid \gamma < \lambda \rangle$ such that:
\begin{itemize}
\item{}
For all $\gamma < \lambda$, $w_\gamma \in \mathcal{W}^1$;
\item{}
$\langle \beta_\gamma \mid \gamma < \lambda \rangle$ is a strictly increasing sequence of ordinals, cofinal in $\alpha$.
\item{}
For all $\gamma < \lambda$, $\Rk(w_\gamma;\mathcal{W})=\beta_\gamma$ (and hence 
$\Rk(\emptyset;\mathcal{W}_{\{w_\gamma \}})= \beta_\gamma + 1$).
\end{itemize}
For each $\gamma < \lambda$, fix, by the inductive hypothesis, $M_\gamma \in K(\mathcal{W}_{\{w_\gamma\} })$, with associated coloring $c_\gamma : [|M_\gamma|]^{<\omega} \rightarrow \mathcal{R}$, with 
$\|M_\gamma \| = \kappa_{\beta_\gamma + 1}$. We may assume that the universes $|M_\gamma|$, $\gamma<\lambda$, are pairwise disjoint. We will now define a structure $M \in K(\mathcal{W})$. The universe of $M$ will be the disjoint union of the universes $|M_\gamma|$ for $\gamma < \lambda$. The coloring $c: [|M|]^{<\omega} \rightarrow \mathcal{R}$ is defined as follows. If $X \in [|M|]^{<\omega}$ and there is $\gamma < \lambda$ such that $X \subseteq |M_\gamma|$, then let $c(X) = c_\gamma(X)$. If there is no such $\gamma$, then let $c(X)$ be an arbitrary element of $\mathcal{R}_{|X|}$. Notice that, in the latter case, there are $x_0, x_1 \in X$ and $\gamma_0 < \gamma_1 < \lambda$ such that $x_0 \in M_{\gamma_0}$ and $x_1 \in M_{\gamma_1}$. In this case, $c(\{x_0\}) = c_{\gamma_0}(1)$ and $c(\{x_1\}) = c_{\gamma_1}(1)$. Since $w_{\gamma_0}$ and $w_{\gamma_1}$ are distinct elements of $\mathcal{W}^1$, we have $w_{\gamma_0}(1) \not= w_{\gamma_1}(1)$. Thus, since $c_\gamma$ is a $\mathcal{W}_{w_\gamma}$-coloring for all $\gamma < \lambda$, we have $c_{\gamma_0}(1) \not= c_{\gamma_1}(1)$, so $X$ cannot be monochromatic. It follows that, if $X \in [M]^{<\omega}$ is monochromatic, then there is $\gamma < \lambda$ such that $X \subseteq M_\gamma$. Then the fact that $M \in K(\mathcal{W})$ follows easily from the fact that each $M_\gamma$ is in $K(\mathcal{W})$.

Now suppose $\alpha = \beta + k$, where $\beta$ is a limit ordinal and $0 < k < \omega$. We first consider the case $\beta = \omega$, $k = 1$. In this case, we may fix $\overline{w} \in \mathcal{W}^1$ with $\Rk(\overline{w};\mathcal{W}) = \omega$ and, in turn, we may fix $\{\overline{w}_n \mid n < \omega \} \subseteq (\mathcal{W}_{\{\overline{w}\}})^2$ such that, for all $m < n < \omega$, $\overline{w}_m \not= \overline{w}_n$. We now construct a model $M$, with associated coloring $c$, in $K(\mathcal{W})$ (in fact in $K(\mathcal{W}_{\{\overline{w}\}})$) such that $\|M\| = \kappa_{\omega+1} = 2^\omega$. The universe of $M$ will be ${^\omega 2}$, the set of all functions $f: \omega \rightarrow 2$. If $f \in {^\omega 2}$, then let $c(\{f\}) = \overline{w}(1)$. If $f \not= g \in {^\omega 2}$, let $\Delta(f,g)$ denote the least $n < \omega$ such that $f(n) \not= g(n)$, and let $c(\{f,g\}) = \overline{w}_{\Delta(f,g)}(2)$. This coloring ensures that no triple of distinct functions $\{f,g,h\}$ is monochromatic for $c$, since it cannot be the case that $\Delta(f,g) = \Delta(f,h) = \Delta(g,h)$. Thus, for $X \in [{^\omega 2}]^{<\omega}$ with $|X| \geq 3$, we may let $c(X)$ be an arbitrary element of $\mathcal{R}_{|X|}$.

Next, suppose $\beta > \omega$ and $k = 1$. Let $\mu = \cf(\beta)$. Fix $\overline{w} \in \mathcal{W}^1$ with $\Rk(\overline{w};\mathcal{W}) = \beta$, and fix an increasing, continuous sequence of ordinals $\langle \alpha_i \mid i < \mu \rangle$ and, for each $i < \mu$, a $\overline{w}_i \in (\mathcal{W}_{\{\overline{w}\}})^2$ such that:
\begin{itemize}
	\item{$\alpha_0 = 0$, $\alpha_1 > \omega$, and $\langle \alpha_i \mid i < \mu \rangle$ is cofinal in $\beta$.}
	\item{For all $i < \mu$, $\Rk(\overline{w}_i;\mathcal{W}) = \alpha_{i+1}$.}
\end{itemize}
For all $i < \mu$, let $\alpha_{i+1} = \beta_{i+1} + k_{i+1}$, where $\beta_{i+1}$ is a limit ordinal and $k_{i+1} < \omega$. By Proposition \ref{pruned_rank}, we may fix $\overline{w}^*_i \in (\mathcal{W}_{\{\overline{w}_i\}})^{k_{i+1}+2}$ with $\Rk(\overline{w}^*_i;\mathcal{W}) = \beta_{i+1}$ and find disjoint sets $\{S^i_j \subset (\mathcal{W}_{\{\overline{w}^*_i\}})^{k_{i+1}+3} \mid j < 2^{k_{i+1}+1}\}$ such that, for all $j < 2^{k_{i+1}+1}$, $\Rk(\emptyset; \mathcal{W}_{S^i_j}) = \alpha_{i+1} + 2$.

For all $i < \mu$ and $j < 2^{k_{i+1}+1}$, let $\mathcal{U}^i_j = (\mathcal{W}_{S^i_j})/\overline{w}$. Note that $\Rk(\emptyset; \mathcal{U}^i_j) = \alpha_{i+1} + 1$. By the inductive hypothesis, fix, for each $i < \mu$ and $j < 2^{k_{i+1}+1}$, a model $M^i_j \in K(\mathcal{U}^i_j)$, with associated coloring $c^i_j: [|M^i_j|]^{<\omega} \rightarrow \mathcal{R}/\overline{w}$, such that $\|M^i_j\| = \kappa_{\alpha_{i+1}}$ (we could do better, but this is sufficient). We may in fact assume that the universe of $M^i_j$ is the interval $[\alpha_i, \alpha_{i+1})$. We will construct a model $M$, with associated coloring $c$, in $K(\mathcal{W}_{\{\overline{w}\}})$, such that $\|M\| = \kappa_{\alpha}$. The universe of $M$ will be ${^\kappa 2}$, where $\kappa := \kappa_\beta$.

If $f \in {^\kappa 2}$, let $c(\{f\}) = \overline{w}(1)$. If $f \not= g \in {^\kappa 2}$, then, as before, let $\Delta(f,g)$ be the least ordinal $\eta < \kappa$ such that $f(\eta) \not= g(\eta)$. For such $f$ and $g$, let $i < \mu$ be such that $\Delta(f,g) \in [\alpha_i, \alpha_{i+1})$, and let $c(\{f,g\}) = \overline{w}_i(2)$. For $i < \mu$, let $\mathcal{A}_i$ be the set of $X \in [^\kappa 2]^{<\omega}$ such that, for all $f \not= g \in X$, $\Delta(f,g) \in [\alpha_i, \alpha_{i+1})$. If $X \in [^\kappa 2]^{<\omega}$, $|X| \geq 3$, and there is no $i < \mu$ such that $X \in \mathcal{A}_i$, then $X$ cannot be monochromatic under $c$, and we can define $c(X)$ to be an arbitrary element of $\mathcal{R}_{|X|}$. If $X \in \mathcal{A}_i$ and $|X| \leq k_{i+1} + 2$, let $c(X) = \overline{w}^*_i(|X|)$. It remains to define $c$ on sets $X$ with $X \in \mathcal{A}_i$ and $|X| > k_{i+1} + 2$.

Let $\prec$ denote the lexicographic ordering of $^\kappa 2$. We will think of elements of $[^\kappa 2]^{<\omega}$ as being finite sets linearly ordered by $\prec$, i.e. sets $\{ f_i \mid i<n \}$ such that, for every $i<n-1$, $f_i \prec f_{i+1}$. If $X=\{ f_i \mid i<n \} \in [^\kappa 2]^{<\omega}$, let $\Delta(X) = \langle \Delta(f_i, f_{i+1}) \mid i<n-1 \rangle$. Note that, if $f_0 \prec f_1 \prec f_2$, then $\Delta(f_0, f_1) \not= \Delta(f_1, f_2)$, since if both quantities were equal to some ordinal $\eta$, then this would imply $f_0(\eta) < f_1(\eta) < f_2(\eta)$, which is impossible.

For all $i < \mu$, enumerate ${^{k_{i+1}+1}2}$ as $\{s^i_j \mid j < 2^{k_{i+1}+1}\}$, where $s^i_0$ is the constant function taking value $0$ and $s^i_1$ is the constant function taking value $1$. If $X \in \mathcal{A}_i$ and $|X| = k_{i+1} + 3$, then define $s_X \in {^{k_{i+1}+1}2}$ by letting, for all $n < k_{i+1}+1$,
\[
s_X(n)= 
\begin{cases}
	0, & \text{if } \Delta(f_n, f_{n+1}) < \Delta(f_{n+1}, f_{n+2});\\
	1, & \text{if } \Delta(f_n, f_{n+1}) > \Delta(f_{n+1}, f_{n+2}).
\end{cases}
\]
In particular, $s_X = s^i_0$ if and only if $\Delta(X)$ is strictly increasing, and $s_X = s_1$ if and only if $\Delta(X)$ is strictly decreasing.

Now, to complete the coloring, if $i < \mu$, $X \in \mathcal{A}_i$, and $|X| = k_{i+1}+3$, find $j$ such that $s_X = s^i_j$. If $j \in \{0,1\}$, then let $c(X) = c^i_j(\Delta(X))$. Note that here we are considering $\Delta(X)$ as a set rather than a sequence, and the fact that $X \in \mathcal{A}_i$ and $\Delta(X)$ is either strictly increasing or strictly decreasing implies that $\Delta(X) \in [[\alpha_i, \alpha_{i+1})]^{k_{i+1}+2}$. If $j > 1$, then choose an arbitrary $Y \in [[\alpha_i, \alpha_{i+1})]^{k_{i+1}+2}$ and let $c(X) = c^i_j(Y)$. Notice that, if $X_0, X_1 \in \mathcal{A}_i$, $|X_0| = |X_1| = k_{i+1}+3$, and $s_{X_0} \not= s_{X_1}$, then $c(X_0) \not= c(X_1)$. If $X \in \mathcal{A}_i$ and $|X| > k_{i+1}+3$, consider $\Delta(X)$. If $\Delta(X)$ is strictly increasing, let $c(X) = c^i_0(\Delta(X))$. If $\Delta(X)$ is strictly decreasing, let $c(X) = c^i_1(\Delta(X))$. Otherwise, let $c(X)$ be an arbitrary member of $\mathcal{R}_{|X|}$.

We must now verify that ${^\kappa 2}$, equipped with this coloring, $c$, is in $K(\mathcal{W})$, i.e. that the diagrams of all monochromatic finite subsets of $^\kappa 2$ are in $\mathcal{W}$. Let $X \in [^\kappa 2]^{<\omega}$. As mentioned above, if there is no $i < \mu$ such that $X \in \mathcal{A}_i$, then $X$ cannot be monochromatic. If $X \in \mathcal{A}_i$ and $|X| \leq k_{i+1} + 2$, then $X$ is monochromatic and its diagram is equal to $\overline{w}^*_i \restriction [|X|]$, which is in $\mathcal{W}$. If $X \in \mathcal{A}_i$, $|X| = k_{i+1} + 3$, and $s_X = s^i_j$, then the diagram of $X$ is in $S^i_j$ and is thus again in $\mathcal{W}$.

It remains to consider the case $X \in \mathcal{A}_i$, $|X| > k_{i+1} + 3$. First, suppose that $X = \{f_n \mid n < n^*\}$ and $\Delta(X)$ is neither strictly increasing nor strictly decreasing. Without loss of generality, there is $m^* < n^* - 3$ such that $\Delta(f_{m^*}, f_{m^*+1}) < \Delta(f_{m^*+1}, f_{m^*+2})$ and $\Delta(f_{m^*+1}, f_{m^*+2}) > \Delta(f_{m^*+2}, f_{m^*+3})$ (the reverse case is handled in the same way). Let $\ell^* = \min(\{m^*, n^* - (k_{i+1}+4)\})$. Let $X_0 = \{f_{\ell^* + k} \mid k < k_{i+1} + 3\}$ and $X_1 = \{f_{\ell^* + k + 1} \mid k < k_{i+1} + 3\}$. Re-enumerate $X_0$ and $X_1$ in lexicographically increasing fashion as $X_0 = \{g_k \mid k < k_{i+1}+3\}$ and $X_1 = \{h_k \mid k < k_{i+1}+3\}$. Notice that, for all $k < k_{i+1}+2$, $h_k = g_{k+1}$ and that, for some $k^* < k_{i+1}+1$, $g_{k^*} = f_{m^*}$. Thus, by our assumptions about $\Delta(X)$, $s_{X_0}(k^*) = 0$. However, $h_{k^*} = f_{m^* + 1}$, so $s_{X_1}(k^*) = 1$. Thus, $X_0, X_1 \in [X]^{k_{i+1}+3}$ and $c(X_0) \not= c(X_1)$, so $X$ is not monochromatic.

Next, suppose that $X \in \mathcal{A}_i$, $X = \{f_n \mid n < n^*\}$, $|X| > k_{i+1} + 3$, and $\Delta(X)$ is strictly increasing. We need the following claim.
\begin{claim}
	If $X$ is monochromatic with respect to $c$, then $\Delta(X)$ is monochromatic with respect to $c^i_0$.
\end{claim}

\begin{proof}
	Suppose $X$ is monochromatic with respect to $c$. It suffices to show that, if $\ell < n^* - 1$ and $D \in [\Delta(X)]^\ell$, then there is $Y \in [X]^{\ell + 1}$ such that $\Delta(Y) = D$. To this end, fix such an $\ell$ and $D$. Let $D = \{\Delta(f_{n_m}, f_{n_m + 1}) \mid m < \ell \}$, where $\{f_{n_m} \mid m < \ell \}$ is enumerated in $\prec$-increasing fashion.

	First note that, under our assumption that $\Delta(X)$ is strictly increasing, if $n < n' < n'' < n^*$, then $\Delta(f_n, f_{n'}) = \Delta(f_n, f_{n''})$. Thus, for all $m < \ell - 1$, $\Delta(f_{n_m}, f_{n_m + 1}) = \Delta(f_{n_m}, f_{n_{m+1}})$, so, if $Y = \{f_{n_m} \mid m < \ell \} \cup \{f_{n^* - 1}\}$, then $Y \in [X]^{\ell + 1}$ and $\Delta(Y) = D$.
\end{proof}

Suppose $X$ is monochromatic with respect to $c$. Then $\Delta(X)$ is monochromatic with respect to $c^i_0$ and thus has diagram $u \in \mathcal{U}^i_0$. But then, by our construction, $X$ has the diagram given by the function $w$ such that $w(n-1) = u(n)$ for $2 \leq n \leq |X|$ and $w(1) = \overline{w}(1)$. But this $w$ is in $\mathcal{W}_{S^i_0}$ and hence in $\mathcal{W}$.

The case in which $X \in \mathcal{A}_i$, $|X| > k_{i+1} + 3$, and $\Delta(X)$ is strictly decreasing is handled in the same way, \emph{mutatis mutandis}.

We finally address the case in which $\beta$ is a limit ordinal and $k > 1$. By Proposition~\ref{pruned_rank}, we may fix $\overline{w} \in \mathcal{W}^k$ with $\Rk(\overline{w};\mathcal{W}) = \beta$ and find disjoint sets $\{S_i\subset (\mathcal{W}_{\{\overline{w}\}})^{k+1} \mid i < 2^{k-1} \}$ such that, for all $i < 2^{k-1}$, $\Rk(\emptyset;\mathcal{W}_{S_i}) = \beta + k$.

%. Note that $\Rk(\emptyset;\mathcal{W}_{\overline{w}}) = \beta +k$.

%%
For all $i < 2^{k-1}$, let $\mathcal{U}_i = (\mathcal{W}_{S_i})/(\overline{w}\restriction [1])$. 
Note that $\Rk(\emptyset; \mathcal{U}_i) = \beta + k - 1 = \alpha - 1$. By the inductive hypothesis, fix, for each $i < 2^{k-1}$, a model $M_i \in K(\mathcal{U}_i)$, with associated coloring $c_i:[|M_i|]^{<\omega} \rightarrow \mathcal{R}/(\overline{w}\restriction [1])$, such that $\|M_i\| = \kappa_{\alpha - 1} =: \kappa$. We may in fact assume that the universe of each $M_i$ is $\kappa$ itself. We will construct a model $M$, with associated coloring $c$, in $K(\mathcal{W}_{\{\overline{w}\}})$. The universe of $M$ will be ${^\kappa 2}$.

If $X \in [^\kappa 2]^{\leq k}$, then let $c(X) = \overline{w}(|X|)$. Enumerate ${^{k-1}2}$ as $\{s_j \mid j < 2^{k-1}\}$, where $s_0$ is the constant function taking value $0$ and $s_1$ is the constant function taking value $1$. As before, if $X\in [^\kappa 2]^{k + 1}$, $X = \{f_i \mid i < k + 1 \}$, then define $s_X \in {^{k-1}2}$ by letting, for all $i < k-1$, 
	\[
	s_X(i) = 
	\begin{cases}
		0, & \text{if } \Delta(f_i, f_{i+1}) < \Delta(f_{i+1}, f_{i+2});\\
		1, & \text{if } \Delta(f_i, f_{i+1}) > \Delta(f_{i+1}, f_{i+2}).
	\end{cases}
	\]

If $X\in [^\kappa 2]^{k+1}$, find $i$ such that $s_X = s_i$. If $i = 0$ or $i = 1$, then let $c(X) = c_i(\Delta(X))$. If $i > 1$, choose an arbitrary $Y \in [\kappa]^k$ and let $c(X) = c_i(Y)$. If $X\in [^\kappa 2]^{<\omega}$ and $|X| > k + 1$, consider $\Delta(X)$. If $\Delta(X)$ is strictly increasing, let $c(X) = c_0(\Delta(X))$. If $\Delta(X)$ is strictly decreasing, let $c(X) = c_1(\Delta(X))$. Otherwise, let $c(X)$ be an arbitrary element of $\mathcal{R}_{|X|}$. The verification that $^\kappa 2$, equipped with this coloring $c$, is in $K(\mathcal{W})$, proceeds along the same lines as the case $\beta > \omega$, $k = 1$.
\end{proof}

\begin{proof}[Proof of Theorem~\ref{Hanf_downside}]
Existence of the model is given by Lemma~\ref{existence_high}, so we suppose that $\Rk(w;\mathcal{W})\ge \beta+1$ for all $w\in \mathcal{W}^1$ and show disjoint amalgamation for $K(\mathcal{W})$ for models of size $\lambda\le \kappa_\beta$.

Suppose that $\{c_1,c_2\}$ is a special $(\lambda, 2)$-system, where $c_i$ is a coloring of $X\cup \{a_i\}$ for $i=1,2$ and $|X|=\lambda$. If $c_1(a_1)\ne c_2(a_2)$, then the function $c_1\cup c_2$ can be extended to a $\mathcal{W}$-coloring of $X\cup \{a_1,a_2\}$ by assigning arbitrary colors to the finite sets of the form $Y\cup\{a_1,a_2\}$ for $Y\subset X$.

Thus, suppose that $c_1(a_1)=c_2(a_2)$. Take $\overline{w}\in \mathcal{W}^2$ such that $\overline{w}(1)$ is equal to the common value of $c_i(a_1)$ and $\Rk(\overline{w};\mathcal{W})\ge \beta$ (the latter is possible by the assumption on
the rank of colorings in $\mathcal{W}^1$).

Let $\mathcal{U}:=\mathcal{W}/\overline{w}$, and let $c^*$ be a $\mathcal{U}$-coloring of $X$. Such a coloring exists because $|X| \leq \kappa_\beta$ and $\Rk(\emptyset;\mathcal{U})\ge \beta$ by Proposition~\ref{quotient_rank}.

To define the coloring $c\supset c_1\cup c_2$, we only need to define $c(\{a_1,a_2\} \cup Y)$ for all $Y \in [X]^{<\omega}$. To this end, let $c(\{a_1, a_2\}) = \overline{w}(2)$, and let $c(Y\cup \{a_1,a_2\}):=c^*(Y)$ for every non-empty finite subset $Y\subset X$. It is easily verified that $c$ is a $\mathcal{W}$-coloring.
\end{proof}

\section{A family $\mathcal{W}_\alpha$ of rank $\alpha<|L|^+$}

We conclude by showing that, for every infinite cardinal $\kappa$ and every $\alpha < \kappa^+$, there is a language $L$ of size $\kappa$ and a set of allowed diagrams $\mathcal{W}$ in $L$ such that $\Rk(w; \mathcal{W}) = \alpha$ for all $w \in \mathcal{W}^1$. (In fact, a single language $L$ will work for all such $\alpha$, namely a relational language with $\kappa$ distinct $n$-ary relations for all $0 < n < \omega$.) This shows that the bound in Proposition~\ref{infinite_rank}(2) is the best possible and that, for every $\lambda<\beth_{\kappa^+}$, there is a coloring class that has the disjoint amalgamation for models of size up to $\lambda$, but fails to have disjoint amalgamation for arbitrarily large models. In particular, for every limit ordinal $\beta$ with $\omega^2 \leq \beta < \kappa^+$ and every $k < \omega$, there is a set of allowed diagrams $\mathcal{W}$ for which $\Rk(w; \mathcal{W}) = \beta + k + 1$ for all $w \in \mathcal{W}^1$. By the results of the previous sections, the coloring class $K(\mathcal{W})$ has disjoint amalgamation for models of size $\leq \beth_{\beta + k}$ but fails to have amalgamation for models of some size $\leq \beth_{\beta + {k + 3 \choose 2}}$.

The family of examples is the same as described in~\cite{BKS}.

\begin{notation} Fix an infinite cardinal $\kappa$ and an ordinal $\alpha$ with $\kappa \leq \alpha < \kappa^+$.  Let $L_{\alpha}$ contain unary predicates $P_{1; \gamma,\alpha}$ for $\gamma\leq \kappa$ and $n$-ary relation symbols $P_{n;\gamma,\beta}$ for $2\leq n <\omega$, $\gamma<\kappa$,  and $\beta\leq \alpha$. 

Let $\mathcal{W}(\alpha)$ be the set of all functions $w:[n]\to L_\alpha$ such that for all $1\le i<j\le n$ if $w(i)=P_{i,\zeta_i,\alpha_i}$ and  $w(j)=P_{j,\zeta_j,\alpha_j}$, then $\alpha_i>\alpha_j$.
\end{notation}

\begin{claim}
For all $n<\omega$ and all $w:[n]\to L_\alpha$ such that $w\in \mathcal{W}(\alpha)$ we have $w(n)=P_{n,\gamma,\beta}$ if and only if $\Rk(w;\mathcal{W})=\beta$.
\end{claim}

\begin{proof}
Easy induction on $\beta$. If $w(n)=P_{n,\gamma,0}$, then there cannot be a function in $\mathcal{W}$ that properly extends $w$, thus $\Rk(w;\mathcal{W})=0$. Conversely, if  $\Rk(w;\mathcal{W})=0$ and $w(n)=P_{n,\gamma,\delta}$ for $\delta>0$, then $w$ can be extended to a function $\overline{w}\in \mathcal{W}(\alpha)$ by letting, for example, $\overline{w}(n+1):=P_{n+1,0,0}$, so the rank of $w$ cannot be 0.

If $w(n)=P_{n,\gamma,\beta+1}$, then every extension $\overline{w}\in \mathcal{W}(\alpha)$ of $w$ satisfies  $\overline{w}(n+1):=P_{n+1,\gamma',\delta}$ for some $\gamma'\le \kappa$ and some $\delta\le \beta$. Thus, the induction hypothesis and the definition of the rank $\Rk$ give that $\Rk(w;\mathcal{W}(\alpha))=\beta+1$. For the converse, if $\Rk(w;\mathcal{W})=\beta+1$ and $w(n)=P_{n,\gamma,\delta}$, then $\delta$ cannot be less than or equal to $\beta$ by the induction hypothesis. If $\delta\ge \beta+2$, then we can define $\overline{w}(n+1):=P_{n+1,0,\beta+1}$, $\overline{w}\in \mathcal{W}(\alpha)$. Then the implication proved above gives $\Rk(\overline{w};\mathcal{W}(\alpha))=\beta+1$, so $\Rk(w;\mathcal{W}(\alpha))\ge \beta+2$, a contradiction.

The case of a limit ordinal $\beta$ is proved by a similar argument.
% limit case: if w(n)=P_{n,\gamma,\beta} for a limit \beta, then for all \delta<\beta there are extensions
% of rank \delta (and all the extensions will have the rank smaller than \beta). So \Rk has to be \beta.
% conversely, if the rank is a limit ordinal \beta, then \delta cannot be smaller; and cannot be greater than beta+1, because then the rank will be too large. 
\end{proof}

\bibliography{coloringref}

\begin{thebibliography}{1}

\bibitem{Ba}
J.~Baldwin.
\newblock {\em Categoricity}, volume~50.
\newblock American Mathematical Soc., 2009.

\bibitem{BKL}
J.~Baldwin, M.~Koerwien, and C.~Laskowski.
\newblock Amalgamation, characterizing cardinals, and locally finite abstract
  elementary classes.
\newblock 2014.
\newblock Preprint.

\bibitem{BKS}
J.~Baldwin, A.~Kolesnikov, and S.~Shelah.
\newblock The amalgamation spectrum.
\newblock {\em The Journal of Symbolic Logic}, 74(03):914--928, 2009.

\bibitem{BKSfix}
J.~Baldwin, A.~Kolesnikov, and S.~Shelah.
\newblock Correction for ``{T}he amalgamation spectrum''.
\newblock \url{http://www.towson.edu/~akolesni/papers/BKS_correction.pdf},
  2013.

\bibitem{Gro}
R.~Grossberg.
\newblock Classification theory for abstract elementary classes.
\newblock {\em Contemporary Mathematics}, 302:165--204, 2002.

\bibitem{GrVaDupcat}
R.~Grossberg and M.~VanDieren.
\newblock Shelah's categoricity conjecture from a successor for tame abstract
  elementary classes.
\newblock {\em The Journal of Symbolic Logic}, 71(02):553--568, 2006.

\bibitem{Lessmannupcat}
O.~Lessmann.
\newblock Upward categoricity from a successor cardinal for tame abstract
  classes with amalgamation.
\newblock {\em The Journal Of Symbolic Logic}, 70(02):639--660, 2005.

\bibitem{Sh394}
S.~Shelah.
\newblock Categoricity for abstract classes with amalgamation.
\newblock {\em Annals of Pure and Applied Logic}, 98(1):261--294, 1999.

\end{thebibliography}

\bibliographystyle{plain}

\end{document}